\documentclass[a4paper, oneside,11 pt]{article}

\usepackage[latin1]{inputenc}
\usepackage[T1]{fontenc}
\usepackage[english]{babel}
\usepackage{amsfonts}
\usepackage{amssymb}
\usepackage{amsmath,mathrsfs}
\usepackage{lmodern}
\usepackage{amsthm}
\usepackage{graphicx}
\usepackage{vmargin}

\usepackage[all]{xy}
\usepackage{verbatim}
\usepackage{amsmath}
\usepackage{amsthm}
\usepackage{graphicx}
\usepackage{amssymb}
\usepackage{epstopdf}
\usepackage{url}
\usepackage{amsfonts}
\usepackage{todonotes}

\def\bR {\mathbf{R}}
\def\bS {\mathbf{S}}

\def\bZ {\mathbf{Z}}

\def\cF {\mathcal{F}}

\newcommand{\tr}{\operatorname{tr}}

\newcommand{\ba}{\begin{aligned}}
\newcommand{\ea}{\end{aligned}}

\newcommand{\be}{\begin{equation}}
\newcommand{\ee}{\end{equation}}

\def\scrL{\mathscr{L}}

\def\eps {{\epsilon}}
\def\th {{\theta}}

\def\om {{\omega}}

\def\d {{\partial}}
\def\grad {{\nabla}}

\newcommand{\mc}{\mathcal}
\newcommand{\pt}{\partial}

\newcommand{\vp}{\varphi}

\newcommand{\br}{\mathbb{R}}

\renewcommand{\om}{\omega}
\newcommand{\te}{\theta}

\renewcommand{\eps}{\varepsilon}
\newcommand{\e}{\varepsilon}

\renewcommand{\r}{\rho}

\renewcommand{\(}{\left(}
\renewcommand{\)}{\right)}
\renewcommand{\[}{\left[}
\renewcommand{\]}{\right]}

\newtheorem{thm}{Theorem}
\newtheorem{lem}[thm]{Lemma}
\newtheorem{cor}[thm]{Corollary}

\newtheorem{remark}[thm]{Remark}

\def\be{\begin{equation}}
\def\ee{\end{equation}}
\def\bea{\begin{eqnarray}}
\def\eea{\end{eqnarray}}

\numberwithin{thm}{section}
\numberwithin{equation}{section}
\numberwithin{figure}{section}

\title{A gradient flow perspective on the quantization problem}

\author{
   Mikaela Iacobelli\
   \thanks{Durham University, Department of Mathematical Sciences, Stockton Road, DH1 3LE, Durham, UK. Email: \textsf{mikaela.iacobelli@durham.ac.uk}}
}

\begin{document}

\maketitle

\begin{abstract}
In this paper we review recent results by the author on the problem of quantization of measures. More precisely, we propose a dynamical approach, and we investigate it in dimensions 1 and 2. Moreover, we discuss a recent general result on the static problem on arbitrary Riemannian manifolds.
\end{abstract}


\section{Introduction}

The term \emph{quantization} refers to the process of finding an \emph{optimal} approximation of a $d$-dimensional probability density by a convex combination of a finite number $N$ of Dirac masses. The quality of such approximation is usually measured in terms of the Monge-Kantorovich or Wasserstein metric.

The need for such approximations first arose in the context of information theory in the early '50s. The idea was to see the quantized measure as the digitization of an analog signal intended for storage on a data storage medium or transmitted via a channel \cite{BW, GG}. Another classical application of the quantization problem concerns numerical integration, where integrals with respect to certain probability measures need to be replaced by integrals with respect to a good discrete approximation of the original measure \cite{GPP}. 
Moreover, this problem has applications in cluster analysis, materials science (crystallization and pattern formation \cite{BPT}), pattern recognition,
speech recognition, stochastic processes, and mathematical models in economics \cite{BS, BJR, S2015} (optimal location of service centers). 
Due to the wide range of applications aforementioned, the quantization problem has been studied with several completely different techniques, and a comprehensive review on the topic goes beyond the purposes of this paper. Nevertheless, it is worth to mention that the problem of the quantization of measure has been studied with a $\Gamma$-convergence approach in \cite{BJR, BJR2, BBSS, MT}.
For a detailed exposition on the quantization problem and a complete list of references see the monograph \cite{GL} and \cite[Chapter $33$]{Gr07}. 

\subsection{A motivating example}

{\bf Question}: what is the ``optimal'' way to locate $N$ clinics in a region $\Omega$ with population density $\rho$?

\begin{figure}[h!]
\center
\includegraphics[scale=0.23]{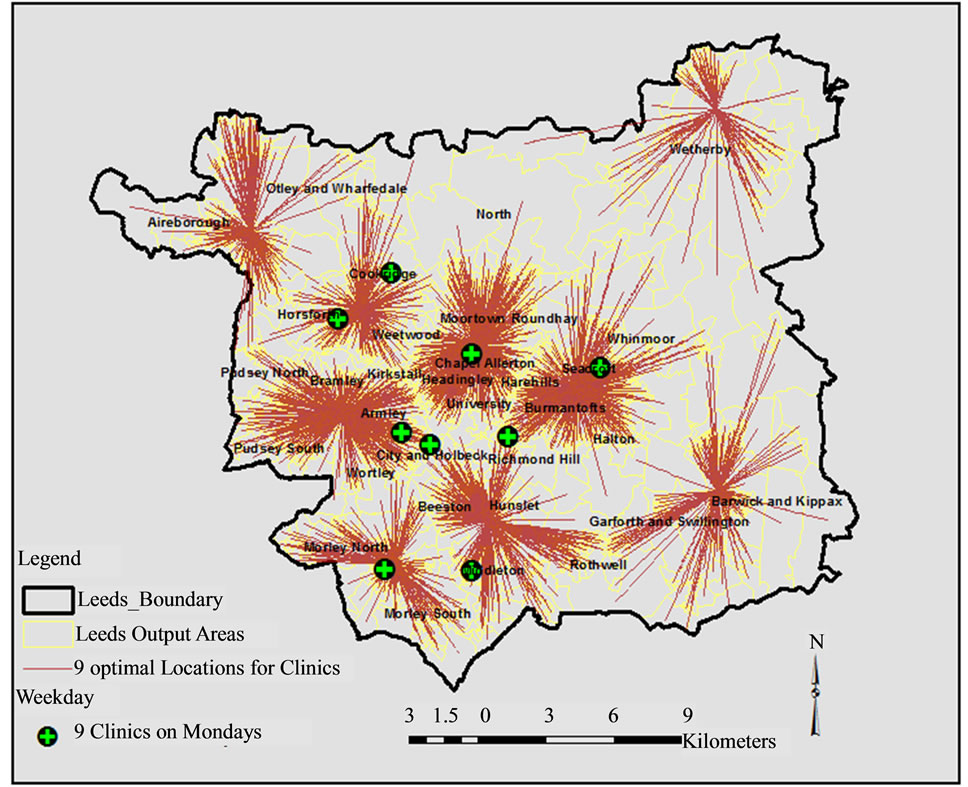}
\caption{Optimal Location of Smoking Cessation Services. Image from \cite{AA}.}
\label{Leeds}
\end{figure}
\smallskip

To answer this question we have to choose:
\begin{itemize}
\item A suitable notion of ``optimality'';
\item the location of each clinic $x_i$;
\item the capacity of each clinic $m_i.$ 
\end{itemize} 

\subsection{Setup of the problem}

We now introduce the theoretical setup of the problem.
Given $r\ge 1$, consider $\rho$ a probability density on an open set $\Omega \subset \br^d$ with finite $r$-th moment,
$$
\int_{\Omega}|y|^r \rho(y)dy<\infty.
$$
Given $N$ points $x^{1}, \ldots, x^{N} \in \Omega,$ we seek the best approximation of $\rho,$ in the sense of Wasserstein distances\footnote{Equivalently known as Monge-Kantorovich distances; we shall use both terms interchangeably. }, by a convex combination of Dirac masses centered at $x^{1}, \ldots, x^{N}:$ 

$$
 W_r\Big(\rho,\sum_im_i \delta_{x^i}\Big)^r:=
\underset{\gamma}{\inf}\bigg\{
\int_{\Omega\times\Omega}|x-y|^rd\gamma(x,y)\,:\,
(\pi_1)_\#\gamma=\sum_im_i \delta_{x^i},\  (\pi_2)_\#\gamma=\rho(y)dy\bigg\},
$$
where $\gamma$ varies among all probability measures on $\Omega \times \Omega$, and $\pi_i: \Omega \times \Omega \to \Omega$ ($i=1,2$) denotes the canonical projection onto the $i$-th factor
 (see \cite{AGS, S2015} for more details on the Monge-Kantorovitch distance between probability measures).

\begin{remark}{\rm We note the following equivalent definition, which the reader may find more intuitive.
Since $\rho$ is absolutely continuous, it follows by the general theory of optimal transport (see for instance \cite{AGS}) that the Wasserstein distance can also be obtained as an infimum over maps: 
 $$
 W_r\Big(\rho,\sum_im_i \delta_{x^i}\Big)^r:=\inf
\int_{\Omega}|y-T(y)|^r\rho(y)\,dy
$$
where $T:\Omega\to\Omega$ varies among all maps that transport $\rho$ onto 
$\sum_im_i \delta_{x^i}$.
In other words, the transport map $T$ partitions a region $\Omega$ with population density $\rho$ into $N$ regions, $\{T^{-1}(x_i)\}_{i=1}^N$. Region $T^{-1}(x_i)$ is assigned to the reource (e.g., clinic) located at point $x_i$ of mass $m_i.$ If $T$ is an \emph{optimal} transport map, then it minimize the $L^r$ distance between the population and the resources (see Figure \ref{transport}).
\begin{figure}[h!]
\center
\includegraphics[scale=0.5]{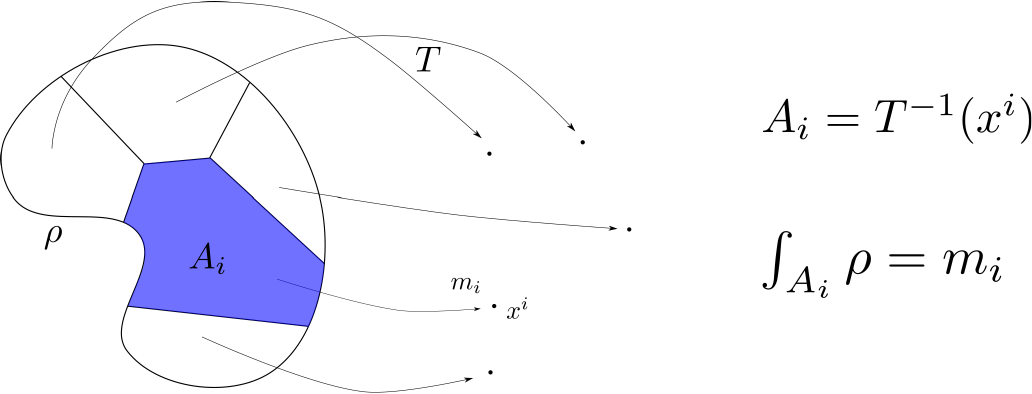}
\caption{Transport map}
\label{transport}
\end{figure}
}\end{remark}
Hence, we  minimize
$$
\inf \bigg\{ W_r\bigg(\sum_im_i \delta_{x^i}, \rho(y)dy\bigg)^r\,:\, m_1, \ldots, m_N\ge0, \  \sum_{i=1}^Nm_i=1
\bigg\}.
$$
As shown in \cite{GL}, the following facts hold:
\begin{enumerate}
\item
The best choice of the masses $m_i$ is given by
$$
m_i:= \int_{W(x^i|\{x^{1}, \ldots, x^{N}\} )} \rho(y)dy,
$$
where 
$$
W(x^i|\{x^{1}, \ldots, x^{N}\}):= \{y \in \Omega\ : \ |y-x^i| \le |y-x^{j}|,\  j \in 1, \ldots, N \}
$$ 
is the so called  \emph{Voronoi cell} of $x^i$ in the set $x^1, \ldots, x^N$
(see Figure \ref{Voronoi}).

\begin{figure}[h!]
\center \includegraphics[scale=0.06]{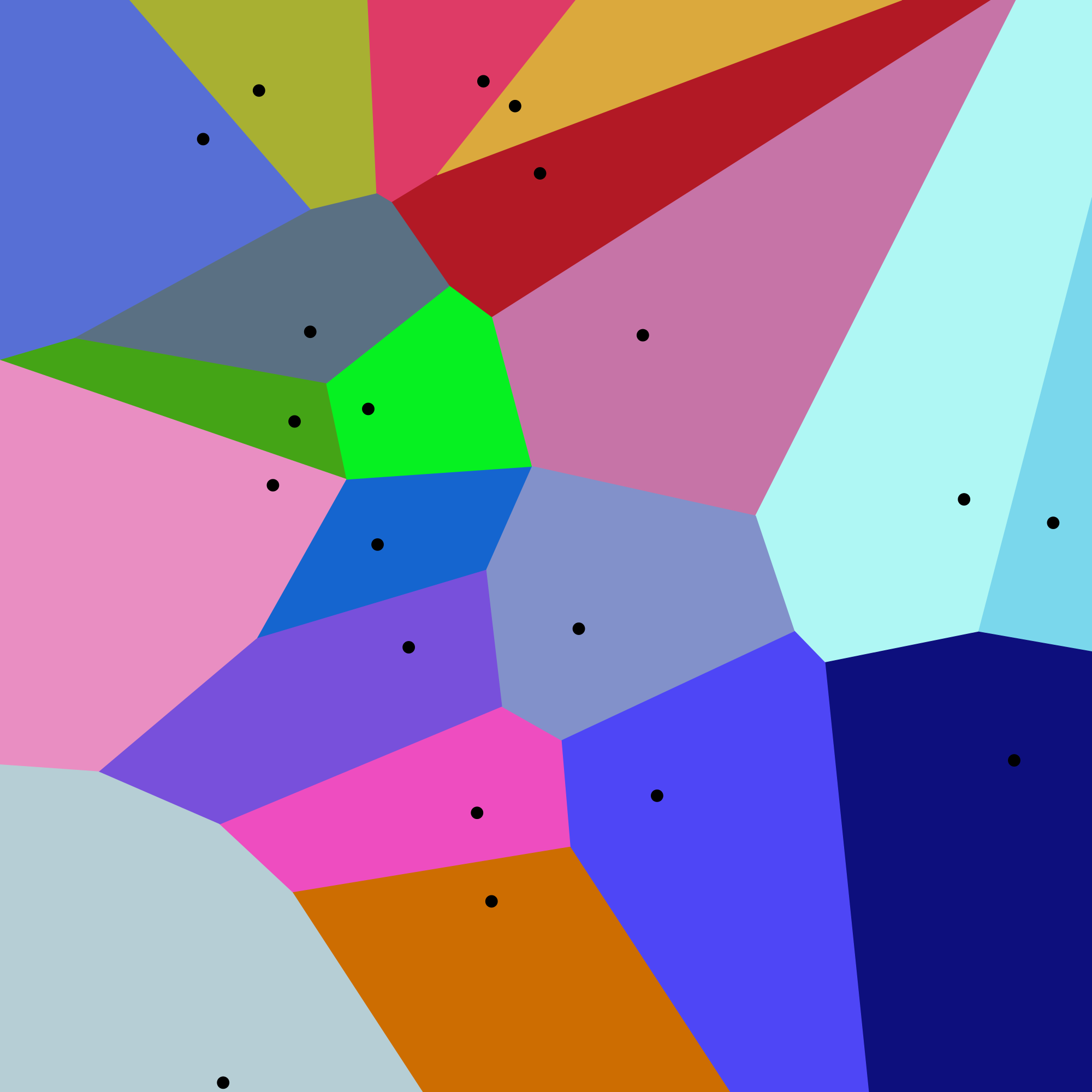}
\caption{$20$ points and their Voronoi cells. \newline Image from Wikipedia \protect\url{https://en.wikipedia.org/wiki/Voronoi_diagram}.}
\label{Voronoi}
\end{figure}

\item The following identity holds:
\begin{multline*}
\inf \bigg\{ MK_r\bigg(\sum_im_i \delta_{x^i}, \rho(y)dy\bigg)\,:\, m_1, \ldots, m_N\ge0, \  \sum_{i=1}^N m_i=1
\bigg\}\\
=F_{N,r} (x^{1}, \ldots, x^{N}) ,
\end{multline*}
where
 $$
F_{N,r} (x^{1}, \ldots, x^{N}) := \int_\Omega \underset{1\le i \le N}{\mbox{min}} | x^i-y |^r\rho(y)dy.
$$
\end{enumerate}

Now the found the optimal masses in terms of $x^{1}, \ldots, x^{N}$, we seek for the optimal location of these points by minimizing $F_{N,r}$. As shown in \cite[Chapter 2, Theorem 7.5]{GL},
if one chooses $x^{1}, \ldots, x^{N}$ in an optimal way by minimizing the functional $F_{N,r}: (\br^d)^N \to \br^+,$ then
in the limit as $N$ tends to infinity these points distribute themselves according to a probability density proportional to $\rho^{d/{d+r}}. $ 
More precisely, under the assumption that
\begin{equation}
\label{moment rho}
\int_{\br^d}|x|^{r+\delta}\rho(x)\,dx<\infty \qquad \text{for some $\delta>0$}
\end{equation}
one has
\be\label{flow1:close}
\frac{1}{N}\sum_{i=1}^N \delta_{x^i} \rightharpoonup \frac{\rho^{d/{d+r}}}{\int_\Omega \rho^{d/{d+r}}(y)dy}\,dx\qquad \text{weakly in $\mathcal P(\Omega)$}.
\ee
These issues have been extensively studied from the point of view of the calculus of variations \cite[Chapter 1, Chapter 2]{GL}.
In \cite{CGI}, we considered a gradient flow approach to this problem in dimension $1$.
Now we will explain the general heuristic of the dynamical approach, and we will later discuss the main difficulties in extending this method to higher dimension.
\smallskip

\subsection{A dynamical approach to the quantization problem.}
Given $N$ points $x^{1}_0, \ldots, x^{N}_0$ in $\br^d$, we consider their evolution under the gradient flow
generated by $F_{N,r}$, that is, we solve the system of ODEs in $(\br^d)^N$
\be
\label{flow1:eq:ODEintro}
\left\{
\begin{array}{rl}
\bigl(\dot x^{1}(t),\ldots,\dot x^{N}(t)\bigr)&=-\nabla F_{N,r}\bigl(x^{1}(t),\ldots,x^{N}(t)\bigr),\\
\bigl(x^{1}(0),\ldots,x^{N}(0)\bigr)&=(x^{1}_0, \ldots, x^{N}_0).
\end{array}
\right.
\ee

As usual in gradient flow theory, as $t$ tends to infinity one expects the points $\bigl(x^{1}(t),\ldots,x^{N}(t)\bigr)$ to converge to a minimizer
$(\bar x^{1},\ldots,\bar x^N)$
of $F_{N,r}.$ Hence, in view of \eqref{flow1:close}, the empirical measure 
$$
\frac{1}{N}\sum_{i=1}^N \delta_{\bar x^i}
$$
is expected to converge to $$\frac{\rho^{d/{d+r}}}{\int_\Omega \rho^{d/{d+r}}(y)dy}dx$$ as $N \to \infty$.

We now want to exchange the limits $t\to \infty$ and $N\to\infty$, and for this we need to take the limit in the ODE above as $N$ goes to infinity.
For this, we take a set of reference points $(\hat x^{1},\ldots,\hat x^N)$ and we parameterize a general family of $N$ points $x^i$ as the image of $\hat x^i$ via a slowly varying smooth map $X:\br^d\to \br^d$, that is
 $$
 x^i=X(\hat x^i).
 $$
In this way, the functional $F_{N,r}(x^{1},\ldots,x^{N})$ can be rewritten in terms of the map $X$, that is
 $$
 F_{N,r}(x^1,\ldots,x^N)=F_{N,r}\bigl(X(\hat x^{1}),\ldots,X(\hat x^{N})\bigr),
 $$
and (a suitable renormalization of it) should converge to
 a functional $\mathcal F[X]$. Hence, we can expect that the evolution of $x^i(t)$ for $N$ large
 is well-approximated by the $L^2$-gradient flow of $\mathcal F$.
Although this formal argument may look convincing,
already the one dimensional case is rather delicate.
In the next section, we review the results of \cite{CGI}.
%
%
%
\section{The 1D case}
The aim of this section is to describe the GF approach introduced above in the one dimensional case.
This case will already show several features of this problem. In particular we will need to study the dynamics of degenerate parabolic equations, and to use several refined estimates on stability of PDEs.

\subsection{The continuous functional}
With no loss of generality let $\Omega$ be the open interval $[0,1]$
and consider $\rho$ a smooth probability density on $\Omega.$
In order to obtain a continuous version of the functional
$$
F_{N,r}(x^{1}, \ldots, x^{N})=\int_{0}^1\underset{1\le i \le N}{\mbox{min}} | x^i-y |^r\rho(y)\,dy,
$$ 
with $0\le x^{1}\le \ldots \le x^{N}\leq 1$, assume that
$$
x^i=X\bigg(\frac{i-1/2}{N}\bigg), \qquad i=1, \ldots, N
$$
with $X: [0,1] \to [0,1]$ a smooth non-decreasing map such that $X(0)=0$ and $X(1)=1$. Then the expression for the minimum becomes
$$
\min_{1\le j\le N}|y-x^{j}|^r=\left\{ \begin{array}{ccc}
|y-x^i|^r & \mbox{for} \ y \in (x^{i-1/2}, x^{i+1/2}),\\
|y|^r & \mbox{for} \ y \in (0, x^{1/2}),\\
|y-1|^r & \mbox{for} \ y \in (x^{N+1/2}, 1),\\
\end{array}
\right.
$$
and $F_{N,r}$ is given by
\begin{multline*}
F_{N,r}(x^{1}, \ldots, x^{N})=\\
 \sum_{i=1}^N \int_{x^{i-1/2}}^{x^{i+1/2}} |y-x^i|^r\rho(y)dy+ \int_{0}^{x^{1/2}} |y|^r\rho(y)dy
+\int_{x^{N+1/2}}^{1} |y-1|^r\rho(y)dy.
\end{multline*}
Hence, by a Taylor expansion, we get
$$
F_{N,r}(x^{1}, \ldots, x^{N})= \frac{C_r}{N^r}\int_0^1 \rho(X(\te))|\pt_\te X(\te)|^{r+1}d\te+ {O}\Big(\frac{1}{N^{r+1}}\Big),
$$
where
$C_r=\frac{1}{2^r(r+1)}$ and ${O}\left(\frac{1}{N^{r+1}}\right)$ depends on the smoothness of $\rho$ and $X$
(for instance, $\rho \in C^1$ and $X \in C^2$ is enough).
Hence 
$$
N^rF_{N,r}(x^{1}, \ldots, x^{N}) \longrightarrow C_r\int_0^1 \rho(X(\te))|\pt_\te X(\te)|^{r+1}d\te:=\mc{F}[X]
$$
as $N\to \infty.$

By a standard computation, we obtain the gradient flow PDE for $\mc{F}$ for the $L^2$-metric,
\begin{multline}\label{flow1:gradient flow}
\pt_tX(t,\te)=C_r\Big((r+1)\pt_\te\big(\rho(X(t,\te))|\pt_\te X(t,\te)|^{r-1}\pt_\te X(t,\te)\big)\\
-\r'(X(t,\te))|\pt_\te X(t,\te)|^{r+1} \Big),
\end{multline}
coupled with the Dirichlet boundary condition
\begin{equation}
\label{flow1:eq:boundary}
X(t,0)=0, \qquad X(t,1)=1.
\end{equation}
\noindent 
{\bf Remark}: in the particular case $\rho \equiv 1,$ we get the $p$-Laplacian equation
$$
\partial_tX=C_r\,(r+1)\,\partial_\theta\big(|\partial_\theta X|^{r-1}\partial_\theta X\big)
$$
with $p-1=r$. Hence, in general, the gradient flow PDE for $\mc{F}$ is a degenerate parabolic equation.
More precisely, the degeneracy comes from the fact that the coefficient $|\partial_\theta X|^{r-1}$ appearing in the equation may vanish or go to infinity.
So a natural question becomes:
\smallskip

\noindent
\textbf{Degeneracy issue:} if $0<c_0\leq \d_\th X_0\leq C_0$, is a similar bound true for all times?

\smallskip
\noindent
Although the answer is easily seen to be positive for the case $\rho\equiv 1$ using that fact that $\partial_\theta X$ solves a ``nice'' equation, the question becomes much more delicate for a general $\rho$. In the next section we show how to give a positive answer to the degeneracy issue for a general class of densities $\rho$.

\subsection{An Eulerian formulation}
Define $f\equiv f(t,x)$ by 
$$
f(t,x)\,dx=X(t,\cdot)_\#d\theta,
$$
namely
$$
\int_0^1 \varphi(x)f(t,x)\,dx=\int_0^1\varphi(X(t,\th))\,d\th \qquad \mbox{for\, all}\, \, \,\varphi \in C^0([0,1]).
$$
Performing the change of variable $x=X(t,\theta)$ in the left hand side, the above identity gives (as long as $X(t,\th):[0,1]\to [0,1]$ is a diffeomorphism)
$$
\int_0^1 \varphi(X(t,\theta))f(t,X(t,\theta))\d_\th X(t,\th)\,d\theta=\int_0^1\varphi(X(t,\th))\,d\th \qquad \mbox{for\, all}\, \, \,\varphi \in C^0([0,1])
$$
from which we deduce (by the arbitrariness of $\varphi$)
$$
f(t,X(t,\th))=\frac1{\d_\th X(t,\th)}.
$$
Then, by a direct computation, we get
\begin{equation}
\label{eulerian}
\left\{
\ba
{}&\d_tf=-r\,C_r\,\d_x\left(f\d_x\left(\frac{\rho}{f^{r+1}}\right)\right)\,,\quad x\in\mathbb R
\\
&f(t,x+1)=f(t,x)
\ea
\right.
\end{equation}

\bigskip
\noindent
{\bf Remark}: if $\rho \equiv 1$ the Eulerian equation becomes
$$
\partial_tf=-C_r\,(r+1)\,\partial_x^2\big(f^{-r}\big)
$$
which is an equation of very fast diffusion type.
\\

Let us set $m:=\rho^{1/(1+r)}$ and $u:=f/m.$ Then the Eulerian quantization gradient flow equation becomes
\begin{equation}
\label{eq:u}
\d_tu=-\frac{(r+1)\,C_r}{m}\d_x\bigg(m\,\d_x\Big(\frac{1}{u^r}\Big) \bigg)\,.
\end{equation}
For the latter equation we can then prove the following comparison principle \cite[Lemma 2.1]{CGI}:

\begin{lem}
If $u>0$ is a solution of \eqref{eq:u} and $c >0$, then
$$
\frac{d}{dt}\int_0^1(u-c)_+(t,x)\,m(x)\,dx\le 0,
$$
$$
\frac{d}{dt}\int_0^1(u-c)_-(t,x)\,m(x)\,dx\le 0.
$$
\end{lem}

Thanks to this lemma, we deduce that the following implication holds for all constants $0<c_0\leq C_0$:
$$
c_0 \leq u(0,x) \le C_0\qquad\Rightarrow \qquad c_0 \leq u(t,x) \le C_0\qquad \mbox{for\, all}\, \, \,t \geq 0.
$$
Therefore, we obtain the following comparison princile:
\begin{cor}
\label{cor:monot}
Assume that $0<\lambda \leq \rho \leq 1/\lambda$ and $0<a_0 \leq \partial_\theta X(0)\leq A_0$.
Then there exist $0<b_0 \leq B_0$, depending only on $\lambda,a_0,A_0$, such that
$$
0<b_0 \leq \partial_\theta X(t)\leq B_0\qquad \mbox{for\, all}\, \, \,t \geq 0.
$$
\end{cor}

\medskip
{\bf Remark:}
The equation \eqref{eulerian} is a very fast diffusion equation that has an interest on its own.
In the paper \cite{I2} we investigated the asymptotic behavior of \eqref{eulerian} and its natural gradient flows structure in the space of probability measures endowed with the Wasserstein distance. 
By using this different approach, one can prove convergence results for  \eqref{eulerian} also in situations that are not covered by the results in \cite{CGI,CGI2}.
Using energy-entropy production techniques, one can prove exponential convergence to equilibrium under minimal assumptions on the data when the functional is not convex in the Wasserstein space. Also, by a detailed analysis of the Hessian of the functional, we can provide sufficient conditions for stability of solutions
with respect to the Wasserstein distance.

\subsection{Main result}
Our main result in \cite{CGI} shows that, under the assumptions that $r=2,$ $\|\rho-1\|_{C^2} \ll 1,$ and  that the initial datum is smooth and increasing, the discrete and the continuous gradient flows
remain {\it uniformly} close in $L^2$ for {\it all} times.
In addition, by entropy-dissipation inequalities for the PDE, we show that the continuous gradient flow
converges exponentially fast to the stationary state for the PDE, which is seen in Eulerian variables to correspond
to the measure $\frac{\rho^{1/3}\,d\theta}{\int \rho^{1/3}}$, as predicted by \eqref{flow1:close}.
We point out that the assumption $r=2$ is not essential, and it is imposed just to simplify some computations so as to emphasize the main ideas.

Our main theorem can be informally stated as follows (we refer to \cite{CGI} for the precise assumptions on the initial data):

\begin{thm}
Assume $r=2,$ $\|\rho-1\|_{C^2} \leq \bar \e$, and let $\big(x^1(t),\ldots,x^N(t)\bigr)$ be the gradient flow of $F_{N,2}$ starting from $\big(x^1_0,\ldots,x^N_0\bigr)$.
Under some suitable assumptions on
$\rho$ and the initial data,
the continuous and discrete GF remain quantitatively close for all times:
$$
\frac1N\sum_{i=1}^N\Big|x_i(N^3t)-X(t,\tfrac{i-1/2}{N})\Big|^2\le\frac{C'}{N^4}\,,\quad t\ge 0.
$$
\vspace{-0.2cm}

In particular

\vspace{-0.2cm}
$$
W_1\bigg(\frac1N\sum_i \delta_{x^i(t)},\frac{\rho^{1/3}\,d\theta}{\int \rho^{1/3}}\bigg) \leq \frac{ 2C'}{N} \qquad \mbox{for\, all}\, \, \,t \geq \frac{N^3\log N}{ c'}.
$$
\end{thm}

We now give a quick overview of the proof of this result, and we refer the reader to \cite{CGI} for a detailed proof.

\begin{proof}[Strategy of the proof]
As we shall explain, the proof in the case $\rho\not\equiv 1$ is more involved than the case $\rho\equiv 1$.
We begin with the simpler case $\rho\equiv 1$.

   \smallskip

\noindent
{\it $\bullet$ The case $\rho\equiv 1$.}
In this situation the $L^2$-GF of $\mathcal F$ depends on $\partial_\theta X$ and $\partial_{\theta\theta}X$,
but not on $X$ itself. 
By a discrete maximum principle for the incremental quotients, we can show that the discrete monotonicity estimate 
$$
\frac{c_0}{N}\leq x^{i+1}(t)-x^i(t)\leq {C_0}{N} \qquad \mbox{for\ all}\, \,i
$$
holds for all times, provided it is satisfied at time $0$.
Thanks to this information, 
we can perform a Gronwall-type argument on the quantity
$$
\frac1N\sum_{i=1}^N\Big|x_i(N^3t)-X(t,\tfrac{i-1/2}{N})\Big|^2,
$$
and this allows us to prove that the 
discrete and the continuous gradient flows remain uniformly close in $L^2$ for {\it all} times.

   \smallskip
\noindent
{\it $\bullet$ The case $\rho\not\equiv 1$.}
This case is more delicate because there is no clear way to show the validity of 
the discrete monotonicity estimate, so the approach for the case $\rho\equiv 1$  completely fails.
To circumvent this, we implement a bootstrap argument 
   that combines a finite-time stability in $L^\infty$ with $L^2$ exponential convergence. 
   This is roughly described in the next 5 steps.
   
   \smallskip
  \noindent
  {\bf Step 1:} We show that
   $$
   \hat X(t):=\Bigl(X\Big(t,\tfrac{1/2}{N}\Big), \ldots, X\Big(t,\tfrac{N-1/2}{N}\Big)\Bigr)
   $$
   solves the discrete gradient flow equation up 
   to an error of order $1/N^2$.

      \smallskip
  \noindent
   {\bf Step 2:} We prove that the discrete and continuous gradient flows stay $1/N^2$-close on
   a finite interval of time, namely
   $$
   \Big|x^i(N^3t)-X(t,\tfrac{i-1/2}{N})\Big| =O\bigg(\frac{1+T}{N^2}\bigg)\qquad\mbox{for\, all}\, \, \,i,\,\, \mbox{for\, all}\, \, \,t \in[0,T].
   $$
          
             \smallskip
  \noindent
   {\bf Step 3:} By Step 2, we are able to transfer the discrete monotonicity estimate from $X(t,\tfrac{i}{N})$
   to $x^i(N^3t)$ on $[0,T]$. More precisely, it follows by Corollary \ref{cor:monot} that
   $$
\frac{b_0}{N} \leq  X(t,\tfrac{i+1/2}{N})-X(t,\tfrac{i-1/2}{N})\leq \frac{B_0}{N}\qquad\mbox{for\, all}\, \, \,i,\,\, \mbox{for\, all}\, \, \,t \in[0,T],
   $$
   so a triangle inequality yields 
     $$
\frac{b_0}{2N} \leq  x^{i+1}(t)-x^i(t)\ \leq \frac{2B_0}{N}\qquad\mbox{for\, all}\, \, \,i,\,\, \mbox{for\, all}\, \, \,t \in[0,T],
   $$
   provided $T$ is bounded and $N$ is sufficiently large.
   
                \smallskip
  \noindent
   {\bf Step 4:} Thanks to the monotonicity bound established in Step 3, as in the case $\rho\equiv 1$
   we are now able to perform a Gronwall argument in $L^2$ to deduce that 
   $$
   t\mapsto \frac1N\sum_{i=1}^N\Big|x^i(N^3t)-X(t,\tfrac{i-1/2}{N})\Big|^2
   $$
   decrease exponentially in time on $[0,T]$. For this step, the assumption $\|\rho-1\|_{C^2} \ll 1$ is crucial (see also Section \ref{sec:counterex} below).
   
                 \smallskip
  \noindent
   {\bf Step 5:} This is the key step: choosing $T$ carefully,  for $N$ large enough, the exponential gain from Step 4 allows us to iterate the argument above starting from time $T$ instead of $0$, and obtain the previous estimates on $[T,2T].$ Iterating infinitely many times, this concludes the proof.
\end{proof}
 
\subsection{On the assumptions  $\|\rho-1\|_{C^2} \ll 1$}
\label{sec:counterex}
As we have seen in the previous section, we have been able to prove the closeness of the discrete and continuous gradient flow, together with an exponential stability estimate, 
under the assumption $\|\rho-1\|_{C^2} \ll 1$.
The aim now is to show that the hypothesis $\|\rho-1\|_{C^2} \ll 1$ is {\it necessary} to ensure the convexity of $\mathcal F$ (and therefore to hope to obtain $L^2$-stability).

It will be convenient to specify the dependence of $\mathcal F$ on $\rho$, so we denote
$$
\mathcal F_\rho(X):=\int_0^1 \rho(X)\,|\partial_\theta X|^3\,d\theta.
$$
We begin by computing the Hessian of $\mathcal F_\rho$

Assume $\lambda \leq \rho \leq \frac{1}{\lambda}$, and
let $X,Y \in L^2([0,1])$ with $0\leq c \leq \partial_\theta X\leq C$ and
$|\partial_\theta Y|\leq C$.
Note that, to ensure that $(X+s Y)(0)=0$ and $(X+s Y)(1)=1$ for all $s$ small, we need to 
assume that 
$$
X(0)=0,\qquad X(1)=1,\qquad Y(0)=0,\qquad Y(1)=0.
$$
Then
\begin{align*}
D^2\mathcal F_\rho[X](Y,Y)&=\frac{d^2}{ds^2}|_{s=0}\mathcal F_\rho[X+s Y]\\
&= 6\int_0^1 \rho(X)\,\pt_\te X\,(\pt_\te Y)^2\,d\te\\
& + 6 \int_0^1 \rho'(X)\,(\pt_\te X)^2\,(\pt_\te Y)\,Y\,d\te
+ \int_0^1 \rho''(X)\,(\pt_\te X)^3\,Y^2\,d\te.
\end{align*}
To build a counterexample, we consider $X(t, \te)=\te$.
By the formula for the Hessian above, we see that for any smooth density 
$\bar \rho$ and for any smooth function $Y$,
\begin{align*}
D^2\mathcal F_{\bar\rho}(X)[Y,Y]= 6\int_0^1\bar \rho\,(\pt_\te Y)^2\,d\te
+6\int_0^1\bar \rho'\,\pt_\te Y\,Y\,d\te+\int_0^1\bar \rho''\,Y^2\,d\te.
\end{align*}
Integrating by parts we have
\begin{align*}
D^2\mathcal F_{\bar\rho}(X)[Y,Y]& =6\int_0^1\bar \rho\,(\pt_\te Y)^2\,d\te-6\int_0^1\bar \rho\,(\pt_\te Y)^2-6\int_0^1\bar \rho\, \pt^2_\te Y\, Y\, d\te\\
&\qquad+2\int_0^1\bar \rho\bigg[(\pt_\te Y)^2+\pt^2_\te Y\, Y  \bigg]\, d\te\\
&=2\int_0^1\bar \rho\,(\pt_\te Y)^2\,d\te-4\int_0^1 \bar \rho\,\pt^2_\te Y\, Y\, d\te.
\end{align*}
We now fix $\e\in (0,1/8)$ to be a small number and define
$$
\bar{\rho}(\te):= \left\{ \begin{array}{cc} 1 & \mbox{for}\ \te \in \[\frac{1}{2}-\e,  \frac{1}{2}+\e\]\\
0 & \mbox{for}\ \te \in [0,1] \setminus \[\frac{1}{2}-\e,  \frac{1}{2}+\e\].
\end{array} \right.
$$
Also, let $Y(t, \te)$ be a Lipschitz function, compactly supported in $(0,1)$, that
is smooth on $(0,1/2)\cup(1/2,1)$ and coincides with $|\te-\frac{1}{2}|+1$ in $\[\frac{1}{2}-\e,  \frac{1}{2}+\e\].$

Since $\bar\rho$ and $Y$ are not smooth, we first extend both of them by periodicity to the whole real line and define $\rho_{\delta}:= \bar \rho\, *\, \vp_{\delta}$ and $Y_\delta:=Y\,*\,\vp_{\delta},$ with 
$$
\vp_{\delta}(\te)= \frac{\exp ^{-\frac{|\te|^2}{2\delta}}}{\sqrt{2\pi \delta}}.
$$
Then
$$
D^2\mathcal F_{\rho_\delta}(X)[Y_{\delta},Y_{\delta}] =2\int_0^1 \rho_{\delta}\,(\pt_\te Y_{\delta})^2\,d\te-4\int_0^1 \rho_{\delta}\,\pt^2_\te Y_{\delta}\, Y_{\delta}\, d\te.
$$
Noticing that
$$
\rho_\delta \to \bar \rho\quad \text{in } L^1,\qquad
 \rho_\delta \to 1 \quad \text{uniformly in }[1/2-\e/2,1/2+\e/2],
 $$
 $$
 Y_\delta \to Y  \quad \text{uniformly},\qquad
\pt_\te Y_{\delta}\to \pt_\te Y \quad \text{a.e.},\qquad
\pt^2_\te Y_{\delta}
\rightharpoonup 2\delta_{1/2},
$$
we see that
$$
D^2\mathcal F_{\rho_\delta}(X)[Y_\delta,Y_\delta] \to  
2\int_{\frac{1}{2}-\e}^{\frac{1}{2}+\e}(\pt_\te Y)^2\,d\te-8Y\(\frac{1}{2}\)=4\e-8< 0 \qquad \mbox{as} \  \delta\to 0.
$$
In particular, by choosing $\delta>0$ sufficiently small, we have obtained that the Hessian of $\mathcal F_{\rho_\delta}$ in the direction $Y_\delta$ is negative when $X(\te)=\te$ and $\rho_{\delta}\in C^{\infty}([0,1])$ and satisfies $1\ge \rho_{\delta}>0$.

\section{The 2D case}
Our goal now is to extend the result described above to higher dimensions.
 As a natural first step, we consider the two-dimensional setting.
 The main advantage in this situation is that optimal configurations are known to be asymptotically triangular lattices \footnote{The vertices of the triangular lattice are the centres of a hexagonal tiling.} \cite{GFT,FT,Gr99, Gr,M}. 
Hence, it looks natural to use the vertices of these lattices 
as the reference points $\hat x^i$ used to parameterize our starting configurations.
In this way we obtain a limiting functional $\mathcal F$ that involves not only $\nabla X$
but also its determinant.
Unfortunately, at present there is no general theory for gradient flows of functionals involving the determinant (this is actually
a major open problem in the field). Moreover, as we shall see, our functional depends in a singular way on the determinant, so it cannot be a convex functional. For this reason, we shall consider initial configurations that are small perturbations of the hexagonal lattices and perform a detailed analysis of the linearized equation.
Combining this with some general $\epsilon$-regularity theorems for parabolic systems,
we prove that the nonlinear evolution is governed by the linear dynamics, and in this way we can prove exponential convergence to the hexagonal configurations.

\subsection{Setting of the problem}
To state our main result, let us consider a regular hexagonal Voronoi tessellation $\scrL$ of the Euclidean plan $\bR^2$ with sides of length $1$. 
We consider the triangular regular lattice 
$$
\scrL:=\bZ e_1\oplus\bZ e_2\,,\quad e_1:=(1,0)\,,\quad e_2:=(\tfrac12;\tfrac{\sqrt{3}}2),
$$
and we note that the Voronoi cells for the points in $\scrL$ are regular hexagons.
To increase the number of points, we consider its dilations
$$
\eps\scrL,\qquad \eps>0.
$$
Let 
$$
\Pi:=\{a e_1+b e_2\,:\,|a|\leq 1/2,\,|b| \leq 1/2\}
$$
 be a fundamental domain, and observe that the periodicity of $\Pi$ and $\eps \scrL$ are compatible for any $\eps=1/n$.

To modify the regular hexagonal lattice, we look
at $\Pi$-periodic deformations of $\eps\scrL$ (see Figure \ref{esagoni})
$$
X(\eps\scrL)\,,\quad \eps=1/n\,,\,\,n\in\mathbb N,
$$
where $X\in\hbox{Diff}(\bR^2)$ satisfies
\begin{align*} 
X \quad \text{is }\Pi\text{-periodic},\qquad \|X-{\rm id}\|_{L^\infty}\ll 1.
\end{align*} 
\begin{figure}[h!]
\center
\includegraphics[scale=0.6]{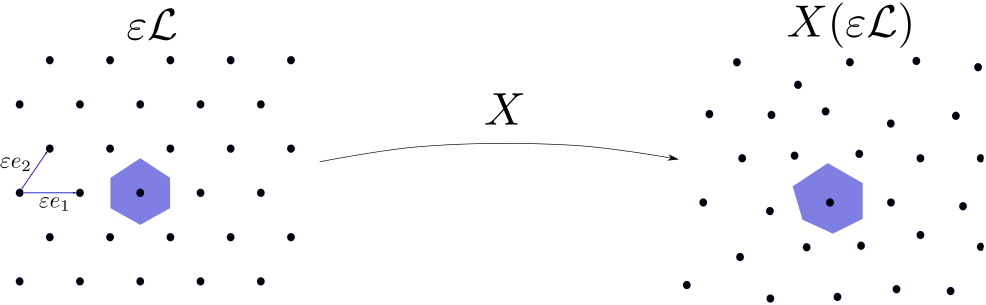}
\caption{$\Pi$-periodic deformations of $\eps\scrL$}
\label{esagoni}
\end{figure}

Note that, up to a translation, we can assume that
$$
\int_{\Pi}X\,dx=\int_{\Pi}{\rm id}\,dx=0.
$$
Our goal is to compute the energy $\mathcal F$ of $X$ as $\eps=1/n \to 0$, and prove that,
under the gradient flow of $\cF$, the limit of the near-hexagonal Voronoi tesselation of $X(\scrL/n)$ converges to the regular hexagonal tesselation.

\subsection{The continuous functional}
Let $(x_1^n,\ldots,x_N^n)=X(\scrL/n)\cap \Pi$ and consider the functional $F_{N,2}(x_1^n,\ldots,x_N^n)$.
By a geometric argument and a delicate computation, we show that\footnote{Note that this corresponds to the quantization of $\rho\equiv 1$ with $d=r=2$ for $N\approx n^2\to\infty$.}
$$
F_{N,2}(x_1^n,\ldots,x_N^n)\approx \frac{1}{n^4}\cF[X],
$$
where
$$
\mathcal F[X]=\int_\Pi F(\nabla X)\,dx,
$$
and, for each $M\in M_2(\bR)$ ,
$$
F(M)=\tfrac13\sum_{\om\in\{e_1,e_2,e_{12}\}}|M\cdot\om|^4\Phi(\om,M)\Big(3+\Phi(\om,M)^2\Big)
$$
with
$$
\Phi(\om,M):=\sqrt{\frac{|MR\om|^2|MR^T\om|^2}{\tfrac34{\rm det}(M)}-1}
$$
for each $\om\in\bS^2$, and
$$
R:=\left(\begin{matrix}\tfrac12\,\,&-\tfrac{\sqrt{3}}2\\ \\ \tfrac{\sqrt{3}}2&\tfrac12\end{matrix}\right),
$$
$$
e_1=(1,0),\quad e_2=Re_1,\quad e_{12}=R^{-1}e_1=e_1-e_2.
$$
Hence the gradient flow is given by 
$$
\d_tX(t,x)={\rm div}\big(\grad F(\grad X(t,x))\big)
$$
with initial and boundary conditions
$$
\left\{
\begin{array}{ll}&X(t) \quad \text{is }\Pi\text{-periodic},
\\
&X(0)=X^{in}.
\end{array} 
\right.
$$
Particularly useful for our analysis is the following more manageable formula:
\begin{align*}
F(M)&:=\frac{1}{16\sqrt{3}}\det(M) \tr[M^TM(2S-I)]\\
&\quad+\frac{1}{64\sqrt{3}}\frac{[\tr(M^TM)]^2[\tr(M^TMS)]}{\det(M)}\\
&\quad\quad-\frac{1}{192\sqrt{3}}\frac{[\tr(M^TM)]^3 +4[\tr(M^TMS)]^3}{\det(M)},
\end{align*}
where
$$
S=\left(\begin{array}{ll}
1 & 0 \\
0 & -1
\end{array}\right).
$$
Note that $F$ depends on $\det (M)$, and blows up as $\det (M)\to 0$.
In particular this implies that $F$ cannot be convex.

\subsection{The small deformation regime}
\label{sec:small}
As mentioned in the introduction, there is no existence theory for gradient flows depending in a singular way on the determinant.
For this reason, it makes sense to focus on a perturbative regime. Hence we write $X={\rm id}+\tau Y$ with $|\tau|\ll 1$, and compute 
\begin{align*}
3\sqrt{3}&\,F({\rm Id}+\tau \nabla Y)=10+{ 20\,\tau\,\rm div}(Y)
\\
&+\tau^2({ 14\,{\rm det}(\nabla Y)}+10\,{\rm div}(Y)^2+3\,|\nabla Y|^2)+O(\tau^3).
\end{align*}
We note that, by the expansion above, one can see that the function $F:\mathbb R^2\times \mathbb R^2\to \mathbb R$ is not convex.
Luckily the following crucial fact holds as a consequence of the fact that $Y$ is periodic:
$$
\int_\Pi {\rm div}(Y)=\int_\Pi {\rm det}(\nabla Y)=0.
$$
Thus, if we set 
$$
F_0(A)=F(A)-\tfrac{20}{3\sqrt{3}}{\rm Tr}(A-{\rm Id})-\tfrac{14}{3\sqrt{3}}{\rm det}(A-{\rm Id}),
$$
then $F_0$ is uniformly convex if $|A-{\rm Id}|\leq \eta\ll 1.$

As a consequence of these two facts, we deduce that $\cF[X]$ can be rewritten as
\begin{equation}
\label{eq:F0}
\cF[X]=\int_\Pi F_0(\grad X)\,dx,
\end{equation}
and $\cF$ is uniformly convex on functions that are sufficiently close to the identity in $C^1$.

\subsection{Main result}

Our main theorem shows that the hexagonal lattice is asymptotically optimal and dynamically stable:

\begin{thm}
\label{thm:main}
Consider an initial datum such that 
$$
\int_{\Pi}X(0)\,dx=0,
\qquad
\|X(0)-{\rm id} \|_{W^{\sigma,p}(\Pi)}\le \eps_0,
$$
with $p>2,$ and $1+2/p<\sigma$.
Assume that $\eps_0$ is small enough.
Then the gradient flow of $\cF$ exists, is unique, and converge exponentially fast to the identity map, that is
$$
\|X(t)-{\rm id}\|_{L^2}\le\|X(0)-{\rm id}\|_{L^2}e^{-\mu t}.
$$
for some $\mu>0$.
\end{thm}
\begin{proof}[Strategy of the proof]
We begin by recalling that $\mathcal F$ can be rewritten as \eqref{eq:F0}, where $F_0$ is uniformly convex in a neighborhood $B_\eta({\rm Id})$ of the identity matrix.

   \smallskip
  \noindent
{\bf Step 1:} 
Let $G_0:\mathbb R^2\times \mathbb R^2\to \mathbb R$ be a smooth uniformly convex function such that
$$
G_0(A)=F_0(A)\qquad \mbox{for\, all}\, \, \,A\text{ s.t.}\,|A-{\rm Id}|\leq \eta/2,
$$
and define
$$
\mathcal G[X]:=\int_\Pi F_0(\grad X)\,dx.
$$
Hence $\mathcal G$ is a convex functional $\mathcal G$ that coincides with $\mathcal F$ on maps that are $C^1$-close to the identity.

   \smallskip
  \noindent
{\bf Step 2:} Since $G$ is convex, there exists a unique gradient flow $Y(t)$ for $\mathcal G$.
Also, again by the standard theory for convex gradient flows, $Y(t)$ converges 
exponentially fast in $L^2$ to ${\rm id}$.


   \smallskip
  \noindent
{\bf Step 3:} 
By the Sobolev regularity on the initial datum and propagation of regularity for short times, we show that
$$
\|\nabla Y(t)-{\rm Id}\|_\infty \leq \eta/4\qquad \mbox{for\, all}\, \, \,t \in [0,t_0]
$$
for some $t_0>0$ small.

   \smallskip
  \noindent
{\bf  Step 4:} Since the gradient flow of $\mathcal G$ is a system,
there is no regularity theory as for classical parabolic equations. Hence, we cannot automatically guarantee that $Y(t)$ is smooth.
To circumvent this difficulty, we exploit the $L^2$ exponential convergence of $Y(t)$ to ${\rm id}$ 
with a delicate $\epsilon$-regularity theorem for parabolic systems in order to show that
$$
\|\nabla Y(t)-{\rm Id}\|_\infty \leq \eta/4 \qquad \mbox{for\, all}\, \, \,t \geq t_0.
$$

   \smallskip
  \noindent
{\bf  Step 5:} Combining Steps 3 and 4 we obtain that 
$$
\|\nabla Y(t)-{\rm Id}\|_\infty \leq \eta/4\qquad \mbox{for\, all}\, \, \,t \geq 0.
$$
Recalling the definition of $\mathcal G$ (see Step 1), this
implies that $\mathcal G=\mathcal F$ in a neighborhood of  $Y(t)$ for all $t\geq 0$, hence
$Y(t)$ {\it is} the gradient flow for $\mathcal F$, and the desired exponential convergence holds.
\end{proof}

\begin{figure}[h!]
\center
\includegraphics[width=5cm]{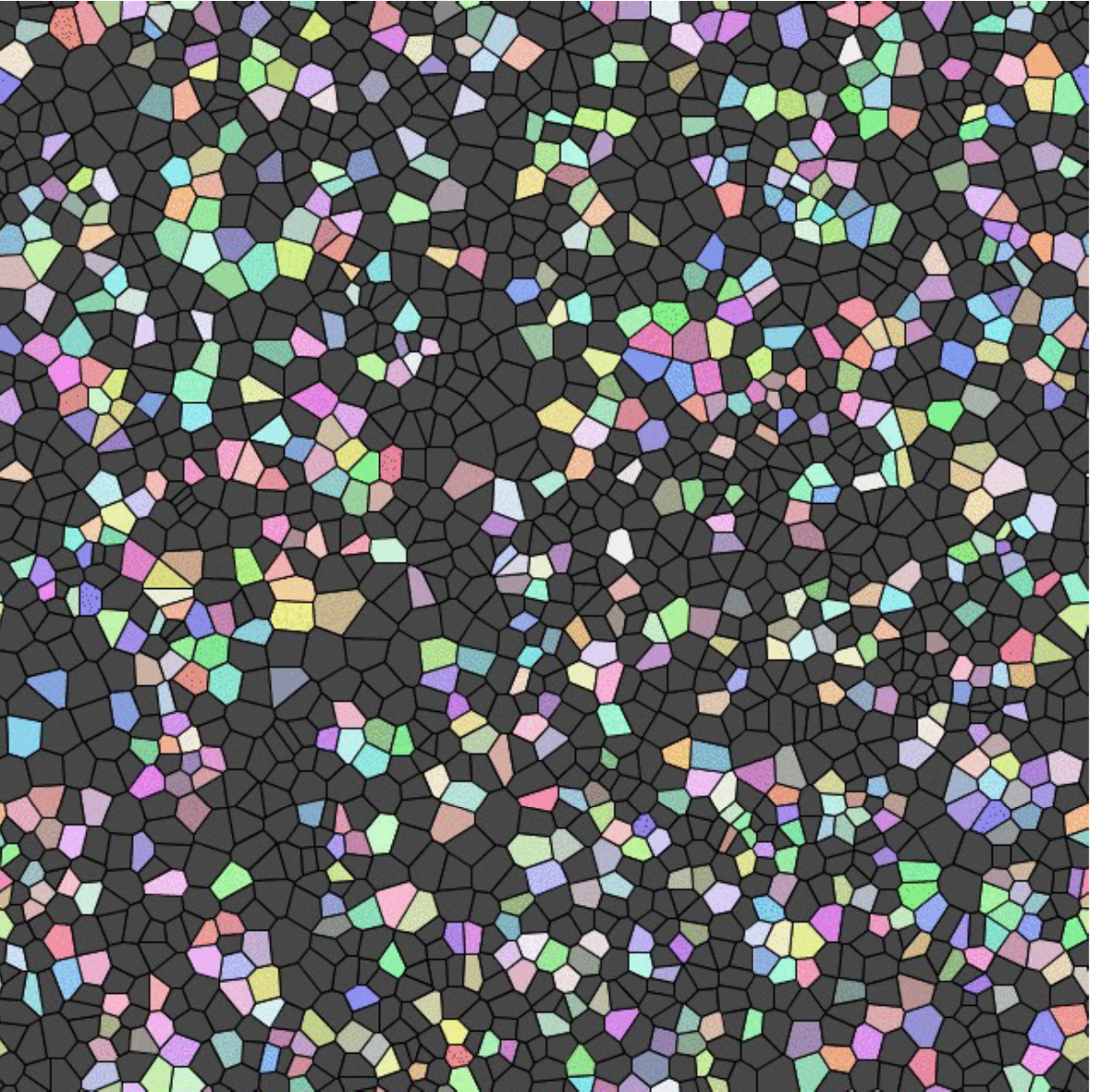}
\caption{720 points at time 0}
\label{F1_emanuele}
\bigskip

\includegraphics[width=5cm]{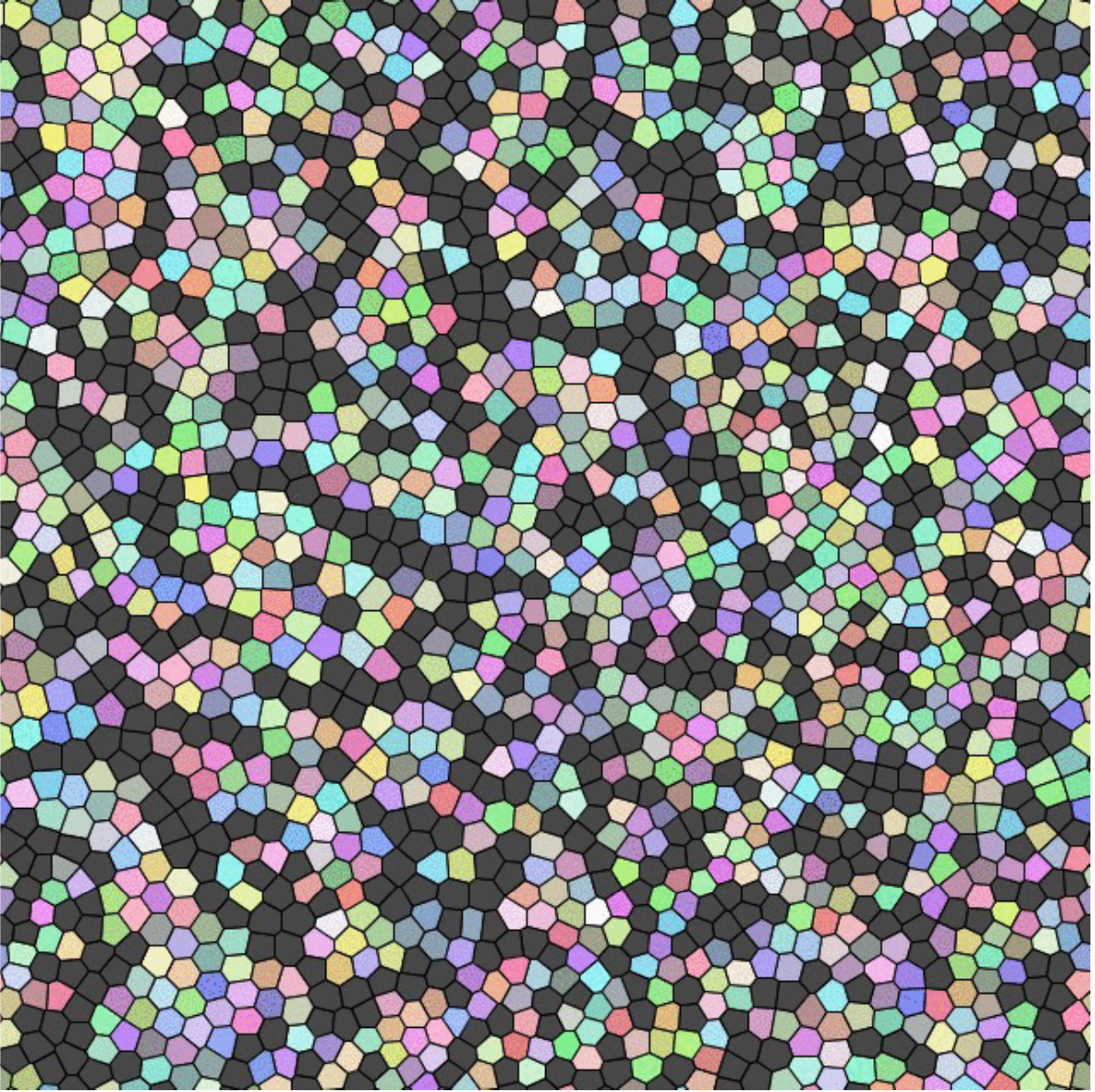}
\caption{720 points after 19 iterations}
\label{F2_emanuele}
\bigskip

\includegraphics[width=5cm]{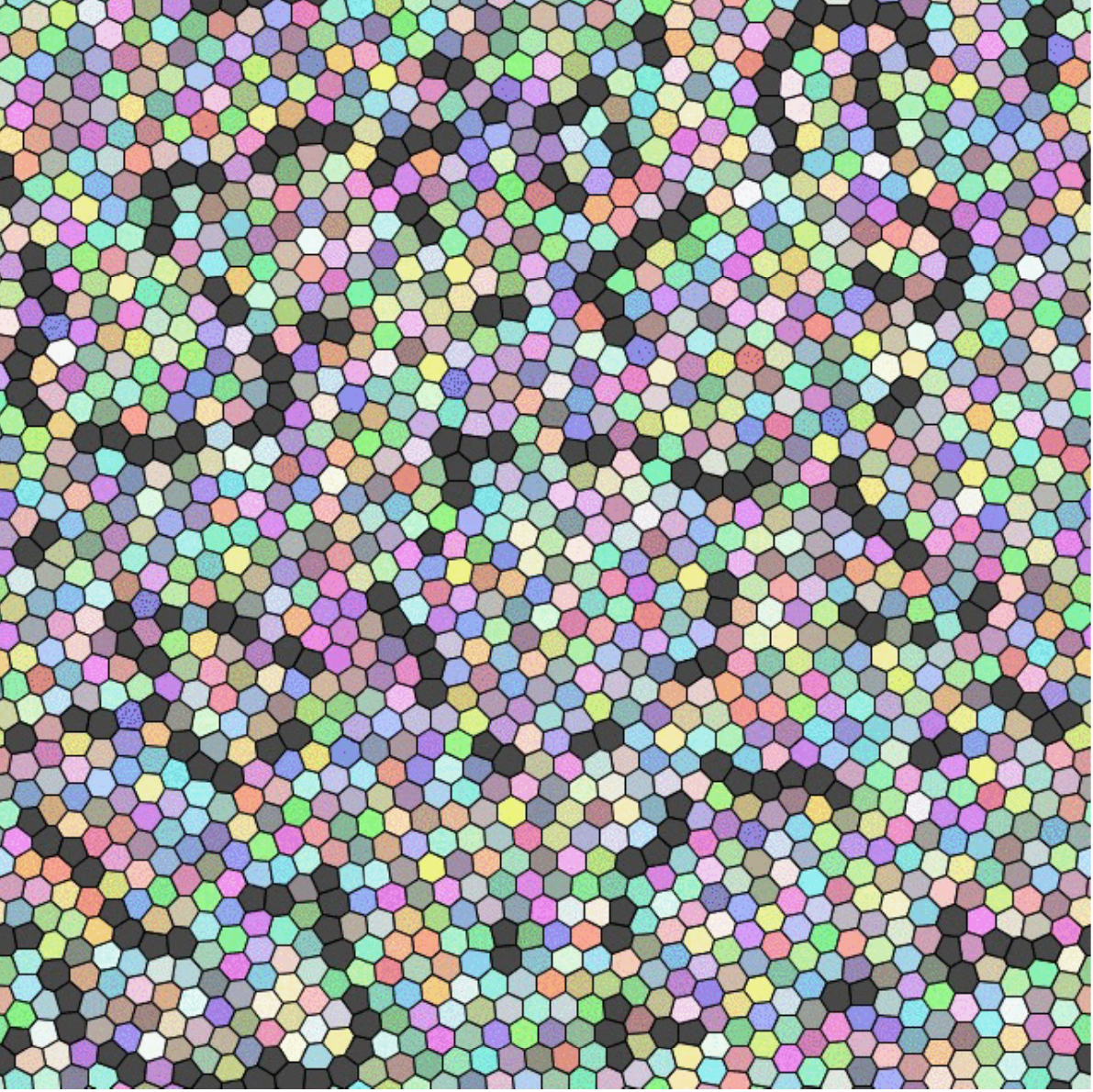}
\caption{720 points after 157 iterations}
\label{F3_emanuele}
\end{figure}

Moreover, our numerical simulations confirm the asymptotic optimality of the hexagonal lattice as the number of points tends to infinity (see Figures \ref{F1_emanuele}, \ref{F2_emanuele}, and \ref{F3_emanuele}). 
Notice that, in Figures \ref{F1_emanuele}, \ref{F2_emanuele}, and \ref{F3_emanuele} the coloured polygons are hexagons.
In Figure \ref{F3_emanuele} it is shown that the minimizers have some small $1$-dimensional defects with respect to the hexagonal lattice. This is due to the fact that the boundary conditions in the simulation are not periodic and on the fact that the hexagonal lattice is not the global minimizer for a finite number $N$ of points.

\newpage
\section{What happens on Riemannian manifolds?}
As described in the introduction, the static version of the quantization problem in $\mathbb R^d$ is well understood.
The aim of this question is to understand what happens when $\mathbb R^d$ is replaced by a Riemannian manifold.
%

In this section we briefly present the results obtained in \cite{I1}.
Our results display how geometry can affect the optimal location problem.

\subsubsection{Main results}
While on compact manifolds one can prove \eqref{flow1:close} by using a suitable \emph{localization argument} (see \cite{I1, B}), the situation is very different when the manifold is \emph{non-compact}. Indeed, some global hypotheses on the behavior of the measure at ``infinity'' have to be imposed. The new growth assumption depends on the curvature of the manifold  and reduces, in the flat case, to a moment condition. We also build an example showing that our hypothesis is sharp.

\smallskip

To state the result we need to introduce some notation:
given a point $x_0 \in \mathcal M$, we can 
consider polar coordinates $(R,\vartheta)$ on $T_{x_0}\mathcal M \simeq \mathbb R^d$ induced by the constant metric $g_{x_0}$,
where $\vartheta$ denotes a vector on the unit sphere $\mathbb S^{d-1}$.
Then, we can define the following quantity that measures the size of the differential of the exponential map when restricted to a sphere
$\mathbb S^{d-1}_R\subset T_{x_0}\mathcal M$:
\begin{equation}
A_{x_0}(R):=R \,\sup_{v\in \mathbb S^{d-1}_R,\,w\in T_{v}\mathbb S^{d-1}_R,\, |w|_{x_0}=1} \Bigl| d_v \exp_{x_0}(w)\Bigr|_{\exp_{x_0}(v)},
\end{equation}
The result on non-compact manifolds reads as follows:
\begin{thm}
\label{thm:noncpt}
Let $(\mathcal M, g)$ be a complete Riemannian manifold, and 
 let $\mu=\rho \,d {\rm vol}$ be a probability measure on $\mathcal M$.

Assume that there exist $x_0 \in \mathcal M$ and $\delta>0$ such that
 \be
 \label{eq:momento}
 \int_{\mathcal M} d(x,x_0)^{r+\delta}\,d\mu(x)+  \int_{\mathcal M} A_{x_0}\bigl({d(x,x_0)}\bigr)^r  \,d\mu(x)<\infty,
\ee
and let $x^{1}, \ldots, x^{N}$ minimize the functional $F_{N,r}: (\mathcal M)^N \to \mathbb{R}^+.$ 
Then \eqref{flow1:close} holds.
\end{thm}

\begin{remark} If $\mathcal M=\mathbb H^d$ is the hyperbolic space, then $A_{x_0}(R)=\sinh R$ and \eqref{eq:momento} reads as
$$
 \int_{\mathbb H^d} d(x,x_0)^{r+\delta}\,d\mu(x)+  \int_{\mathbb H^d} \sinh\bigl({d(x,x_0)}\bigr)^r  \,d\mu(x) \approx \int_{\mathbb H^d} e^{r\,d(x,x_0)} \,d\mu(x)<\infty.
$$
If  $\mathcal M=\mathbb R^d$ then $A_{x_0}(R)=R$ and \eqref{eq:momento} coincides with the finiteness of the $(r+\delta)$-moment of $\mu$, as in \eqref{moment rho}.
\end{remark}
We notice that the moment  condition \eqref{moment rho} required on $\mathbb R^d$ is not sufficient to ensure the validity 
of the result on $\mathbb H^d$. Indeed, as shown in \cite{I1},
there exists a measure $\mu$ on $\mathbb H^2$
such that 
$$
\int_{\mathbb H^2} d(x,x_0)^p\,d\mu<\infty \qquad \mbox{for\, all}\, \, \,p >0,\,\mbox{for\, all}\, \,  \,x_0 \in \mathbb H^2,
$$
but for which the result fails. 
\medskip

 {\it Acknowledgments:} The author would like to thank Megan Griffin-Pickering for her useful comments on a preliminary version of this paper and the L'Or\'eal Foundation for partially supporting this project by awarding the L'Or\'eal-UNESCO \emph{For Women in Science France fellowship}.

\end{document}